\documentclass[reqno,11pt]{amsart}
\usepackage{amsfonts,amsmath,amssymb,fullpage,graphicx}

\def\frak{\mathfrak}
\theoremstyle{plain} \newtheorem{Thm}{Theorem}[section]
\theoremstyle{plain} \newtheorem{Cor}[Thm]{Corollary}
\theoremstyle{plain} \newtheorem{Prop}[Thm]{Proposition}
\theoremstyle{plain} \newtheorem{Lemma}[Thm]{Lemma}
\theoremstyle{definition} 
\theoremstyle{remark} \newtheorem{Rem}[Thm]{Remark}
\theoremstyle{definition} 

\newcommand{\thmlist}{
\renewcommand{\theenumi}{\alph{enumi}}
\renewcommand{\labelenumi}{(\theenumi)}}

\newcommand{\sign}{\mathop{\rm{sign}}}

\newcommand{\grad}{\mathop{\rm{grad}}}

\newcommand{\Hom}{\mathop{\rm{Hom}}}

\newcommand{\diver}{{\rm div\,}}
\newcommand{\norm}[1]{\|\/#1\/\|}

\newcommand{\normuno}[1]{\|\/#1\/\|_1}

\newcommand{\inner}[2]{\langle#1,#2\rangle}

\newcommand{\C}{\ensuremath{\mathbb C}}

\newcommand{\R}{\ensuremath{\mathbb R}}

\renewcommand{\l}{\lambda}
\renewcommand{\a}{\alpha}

\newcommand{\fa}{\mathfrak{a}}

\newcommand{\frakg}{\mathfrak{g}}

\newcommand{\la}{\l_\a}

\begin{document}

\makeatletter
\title[Uncertainty principles for the Schr\"odinger equation]{Uncertainty principles for the Schr\"odinger equation\\ on Riemannian symmetric spaces of the noncompact type}
\author{A. Pasquale}
\address{\textbf{
Laboratoire de Math\'ematiques et Applications de Metz (LMAM, UMR CNRS 7122),
  Universit\'e Paul Verlaine -- Metz, Ile du Saulcy, 57045 Metz cedex 1, France.} \textit{Email:} \textrm{pasquale@math.univ-metz.fr}}
\author{M. Sundari}
\address
{\textbf{Chennai Mathematical Institute, Plot No. H1, SIPCOT IT Park, Padur P.O., Siruseri 603 103, India.}   \textit{Email:} \textrm{sundari@cmi.ac.in} } 
\subjclass[2000]{Primary 43A85; Secondary  58Jxx}
\keywords{Uncertainty principle, Schr\"odinger equation, Helgason-Fourier transform, Beurling theorem, Hardy theorem}
\maketitle

\begin{abstract}
Let $X$ be a Riemannian symmetric space of the noncompact type. We prove that the solution of the time-dependent Schr\"odinger equation on $X$ with square integrable initial condition $f$ is identically zero at all times $t$ whenever $f$ and the solution at a time $t_0>0$ are simultaneously very rapidly decreasing. The stated condition of rapid decrease is of Beurling type. Conditions respectively of Gelfand-Shilov, Cowling-Price and Hardy type are deduced.
\end{abstract}

\section{Introduction}
Consider the initial value problem for the time-dependent Schr\"odinger equation on a Riemannian symmetric space of the noncompact type $X$:
\begin{equation} 
\label{eq:SchroedingerX}
\tag{{\rm S}}
\begin{split}
&i\partial_t u(t,x)+\Delta u (t,x)=0\\
&u(0,x)=f(x)
\end{split}
\end{equation}
where $\Delta$ denotes the Laplace-Beltrami operator on $X$. 
In \cite{Chanillo}, Sagun Chanillo initiated the study of certain uniqueness properties for the solutions of (\ref{eq:SchroedingerX}) that are related with the uncertainty principle. According to this principle, the solution $u(t,\cdot)$ must be identically zero at all times $t$ whenever the initial condition $f$ and the solution at a certain time $t_0\neq 0$ are simultaneously very rapidly decreasing. 
The case considered by Chanillo corresponds to Riemannian symmetric spaces of the form $X=G/K$ where $G$ is a noncompact connected semisimple Lie group possessing a complex structure and $K$ is a maximal compact subgroup of $G$. Moreover, in \cite{Chanillo} the initial condition $f$ is assumed to be $K$-invariant, which ensures that the solution itself is $K$-invariant. Some related articles in the Euclidean setting are
\cite{IK}, \cite{EKPV1}, \cite{EKPV2}, \cite{EKPV3} and \cite{CEKPV}; for the Heisenberg group \cite{BT}.

In this paper, we consider the initial value problem (\ref{eq:SchroedingerX}) on arbitrary Riemannian symmetric spaces $X$ of the noncompact type and not-necessarily $K$-invariant initial data $f \in L^2(X)$. We show that if $f$ and the solution $u(t_0,\cdot)$ at some time $t_0>0$ satisfy a suitable Beurling type condition, then $u(t,\cdot)=0$ for all $t \in \R$. This describes an uncertainty principle because of the  relation one gets between the Fourier transforms of $u_t=u(t,\cdot)$ and the initial condition $f$; see (\ref{eq:FourierHelgasonrelation}). Our result implies in particular the uniqueness property of solutions for the Schr\"odinger equation (\ref{eq:SchroedingerX}) proved by Chanillo using Hardy type conditions.
 
The rapid decrease of a function on $X$ is measured by means of exponential powers of the distance function $d$ induced by the Riemannian metric. If $o=eK$ denotes the base point in $X$, we set $\sigma(x):=d(o,x)$. The estimates are also in terms of the elementary spherical function of spectral parameter $0$, denoted by $\Xi$. We denote by $C(\R:L^2(X))$ the space of continuous functions
$t \mapsto u(t,\cdot)$ from $\R$ to $L^2(X)$.
We refer to section \ref{subsection:invariantfunctions} for the precise definitions. 
Our main result is the following theorem.

\begin{Thm} \label{thm:BeurlingSchr}
Let $X$ be a Riemannian symmetric space of the noncompact type and let $f \in L^2(X)$.
Let $u(t,x) \in C(\R:L^2(X))$ denote the solution of (\ref{eq:SchroedingerX}) with initial condition $f$. If there is a time $t_0> 0$ so that 
\begin{equation}\label{eq:BeurlingSchr}
\int_X \int_X |f(x)| |u(t_0,y)|  \Xi(x) \Xi(y) e^{\frac{\sigma(x)\sigma(y)}{2t_0}} \; dx \,dy < \infty\,, 
\end{equation}
then $u(t,\cdot)= 0$ for all $t \in \R$.
\end{Thm}

The proof of Theorem \ref{thm:BeurlingSchr} is given in section \ref{subsection:BeurlingX}. It is based on a reduction to a Euclidean situation by means of the Radon transform. This is a powerful technique that is commonly used to attack problems related to the uncertainty principle on symmetric spaces as well as on Euclidean spaces; see for instance \cite{Sengupta2000, Sengupta2002, Sarkar-Sengupta-Canadian, SS}. 
In our situation, the important feature is that for functions satisfying (\ref{eq:BeurlingSchr}), the composition of the Radon transform and a Euclidean Fourier transform gives the Helgason-Fourier transform. A Beurling theorem for the Helgason-Fourier transform has been proved by Sarkar and Sengupta (see \cite{SS}, Theorem 4.1), under the assumption $f \in L^1(X) \cap L^2(X)$. However, unlike the Euclidean case treated in section \ref{section:Euclidean} below, we cannot deduce our uniqueness theorem from Beurling's theorem on Riemannian symmetric spaces. In fact, in the Euclidean situation, the moduli of the initial value function $f(x)$ and of the solution $u(t,x)$ are (up to constants) equal to those of a Fourier transform pair $(h,\widehat{h})$. So Beurling's theorem for $(h,\widehat{h})$ yields the required uniqueness for the
Schr\"odinger solution. See Corollary \ref{cor:BeurlingSchrR}. A similar property does not hold for general Riemannian symmetric spaces. This is a consequence of the lack of duality between the Helgason-Fourier transform and its inverse. One can in part recover this duality for $K$-invariant functions on Riemannian symmetric spaces $G/K$ with $G$ complex. In this case, one can work on a maximally flat geodesic submanifold, where the Helgason-Fourier transform reduced essentially to a Euclidean Fourier transform. This is the approach used in \cite{Chanillo}.

The classical version of Beurling's theorem on $\R^n$ is among the strongest qualitative uncertainty principles. By this, we mean theorems which allow to conclude that $f=0$ by giving quantitative conditions on $f$ and its Fourier transform $\widehat{f}$. For instance, Beurling's theorem implies
the qualitative uncertainty principles of Gelfand-Shilov, Cowling-Price and Hardy. Following this line, in section \ref{section:applications} we prove that the uniqueness property stated in Theorem \ref{thm:BeurlingSchr} implies uniqueness properties respectively of Gelfand-Shilov, Cowling-Price and Hardy type, for the solution of the Schr\"odinger equation on a Riemannian symmetric space.

Finally, we study more closely the Hardy type conditions. Let $\alpha$ and $\beta$ be the exponents measuring respectively the very rapid decay of the initial condition $f$ and the solution $u_{t_0}$ at a time $t_0>0$. Then the uniqueness property is proven to hold when the condition $16t_0^2\alpha\beta >1$ is satisfied. The value of $t_0$ is shown to be optimal. More precisely,
in Theorem \ref{thm:complexnonunique} we prove that, under the additional assumption that $G$ admits a complex structure, the $K$-invariant initial conditions $f$ for which the uniqueness property fails 
in the case $16t_0^2\alpha\beta=1$ are precisely the constant multiples of an explicitly given function.

\bigskip

\textit{Acknowledgements:\; }  
This paper was partly written when the authors were visiting the Scuola Normale Superiore di Pisa and the Department of Mathematics
of the IISc Bangalore. The authors would like to express their gratitude to Micheal Cowling and Fulvio Ricci for the invitation to 
the Intensive Research Period ``Euclidean Harmonic Analysis, Nilpotent Lie Groups and PDEs'' and to the Centro di Ricerca Matematica Ennio de Giorgi for financial support. They also thank E.K. Narayanan, A. Sitaram and S. Thangavelu for their hospitality and financial support. The second named author also acknowlegdes financial support from the Chennai Mathematical Institute;
the first named author would like to thank the Indo-French Institute for Mathematics for travel support. Finally, we would like to express our thanks to S. Chanillo, P. Ciatti, M. Cowling, N. Lohou\'e, E.K. Narayanan and F. Ricci for discussions on the subject of this paper. We would also like to thank the referee for a very careful reading of our manuscript, constructive suggestions and some simplifications of our original arguments.

\section{Notation and preliminaries}

\subsection{Riemannian symmetric spaces of the noncompact type and their structure}
Let $X$ be a Riemannian symmetric space of the noncompact type. Then $X=G/K$ where $G$ is a noncompact, connected, semisimple, real Lie group $G$ with finite center and $K$ is a maximal compact subgroup of $G$. Let $\frak g$ be the Lie algebra of $G$ and let $\frak k$ ($\subset \frak g$) be the Lie algebra of $K$. We have the  Cartan decomposition $\frak g=\frak k\oplus \frak p$, where $\frak p$ is the subspace of $\frak g$ which is orthogonal to $\frak k$ with respect to the Killing form $B$ of $\frak g$. Then $\frak p$ can be identified with the tangent space to $X$ at the base point $o=eK$. We fix a maximal abelian subspace $\frak a$ of $\frak p$. We denote by $\frak a^*$ the (real) dual space of $\frak a$ and by $\frak a_\C^*$ its complexification. The Killing form $B$ is a (positive definite) inner product on $\frak p$ and hence on $\frak a$. For $H_1,H_2 \in \frak a$ we set $\inner{H_1}{H_2}:=B(H_1,H_2)$ and 
$|H_1|:=\inner{H_1}{H_1}^{1/2}$. We extend this inner product to $\frak a^*$ by duality by setting
$\inner{\lambda}{\mu} := \inner{H_\lambda}{H_\mu}$ if $H_\lambda$ is the unique element of $\frak a$ so that $\inner{H}{H_\lambda} = \lambda(H)$ for all $H \in \frak a$. The $\C$-bilinear extension of 
$\inner{\cdot}{\cdot}$ to $\frak a_\C^*$ will be indicated by the same symbol. 

The set of (restricted) roots of the pair $(\frak g,\frak a)$ is
denoted by $\Sigma$. It consists of all $\alpha \in \frak a^*$ for
which the vector space $\frak g_{\a} := \{X \in \frak g:
\text{$[H,X]=\a(H)X$  for every $H \in \frak a$}\}$ contains nonzero
elements. The dimension $m_{\a}$ of $\frakg_{\a}$ is called the
multiplicity of the root $\alpha$. We shall adopt the convention that 
$m_{2\a}=0$ if $2\alpha$ is not a root. The subset of $\frak a$ on which all
roots vanish is a finite union of hyperplanes. We can therefore
choose $Y \in \frak a$ so that $\alpha(Y)\neq 0$ for all $\alpha \in
\Sigma$. Put $\Sigma^+:=\{\alpha\in\Sigma:\alpha(Y)>0\}$. Then 
$\Sigma^+$ is a system of positive roots, and $\Sigma$ is the
disjoint union of $\Sigma^+$ and $-\Sigma^+$.  
A root $\a\in \Sigma$ is said to be indivisible if $\a/2 \notin
\Sigma$. We denote by $\Sigma_0$ the set of indivisible roots and 
set $\Sigma_0^+:=\Sigma^+ \cap \Sigma_0$.
The half-sum of the positive roots counted with
multiplicites is denoted by $\rho$: hence
\begin{equation}\label{eq:rho}
\rho=\frac{1}{2}\, \sum_{\alpha\in\Sigma^+} m_\alpha \alpha\,.
\end{equation}

The set $\frak a^+:=\{H\in \fa: \text{$\alpha(H)>0$ for all $\alpha\in
\Sigma^+$}\}$ is an open polyhedral cone called the positive Weyl
chamber. The corresponding Weyl chamber in $\frak a^*$ is 
$\frak a^*_+:=\{\l\in \frak a^*: \text{$\inner{\l}{\a}>0$ for all $\alpha\in
\Sigma^+$}\}$.

Let $\exp:\frak g \to G$ be the exponential map of $G$, and set 
$A := \exp \frak a$ and $A^+ := \exp \frak a^+$. Then $A$ is an abelian subgroup of $G$ that is 
diffeomorphic to $\frak a$ under $\exp$. 
The inverse diffeomorphism is denoted by $\log$. 

The Weyl group $W$ of the pair $(\frak g,\frak a)$ is the finite group of
orthogonal transformations of $\frak a$ generated by the reflections 
$r_\a$ with $\a\in \Sigma$, where
\begin{equation*}
r_\a(H):=H-2\frac{\a(H)}{\inner{\a}{\a}}H_\a\,, \qquad H \in \frak a\,.
\end{equation*}
The Weyl group action extends to $A$ via the exponential map, to $\frak a^*$ by duality,
and to $\frak a_\C^*$ by complex linearity.

Set $\frak n:=\oplus_{\a\in \Sigma^+} \frak g_\a$. 
Then $N := \exp \frak n$ is a simply connected nilpotent subgroup of $G$. The 
subgroup $A$ normalizes $N$. 
The map $(k, a, n) \mapsto kan$ is an analytic diffeomorphism of the product 
manifold $K \times A \times N$ onto $G$.
The resulting decomposition $G = KAN=KNA$ is called the Iwasawa decomposition 
of $G$. Thus, for $g \in G$
we have $g=k(g) \exp H(g) n(g)$ for uniquely determined $k(g)\in K$, $H(g) \in \frak a$ 
and $n(g)\in G$. 
One can prove that 
\begin{equation}\label{eq:estimateH}
 |H(ak)| \leq  |\log a|
\end{equation}
for all $a \in A$ and $k \in K$. 

Let $M$ be the centralizer of $A$ in $K$ and set $B = K/M$. 
Then the map $A : X \times B \to \frak a$ is defined by 
\begin{equation}\label{eq:A}
A(gK, kM) := -H(g^{-1}k)\,. 
\end{equation}
Besides the Iwasawa decomposition, we will also need the polar decomposition $G=KAK$: 
every $g\in G$ can be written in the form $g=k_1 a k_2$ with $k_1,k_2 \in K$ and $a \in A$.
The element $a$ is unique up to $W$-invariance.

\subsection{Normalization of measures and integral formulas}
We shall adopt the normalization of measures as in \cite{He3}, Ch. II, \S 3.1. In particular, the Haar 
measures $dk$ and $dm$ on the compact groups $K$ and $M$ are normalized to have total mass 1.
The Haar measures $da$ and $d\l$ on $A$ and $\frak a^*$, respectively, are normalized so that 
the Euclidean Fourier transform 
\begin{equation} \label{eq:FourierA}
(\mathcal F_Af)(\l):=\int_A f(a) e^{-i\l(\log a)} \; da\,, \qquad \l\in\frak a^*\,,
\end{equation}
of a sufficiently regular function $f:A \to \C$ is inverted by 
\begin{equation} \label{eq:invFourierA}
f(a)=\int_{\frak a^*} (\mathcal F_Af)(\l) e^{i\l(\log a)} \; d\l\,, \qquad a\in A\,.
\end{equation}
The Haar measures $dg$ and $dn$ of $G$ and $N$, respectively, are normalized so that 
$dg=e^{2\rho(\log a)} \; dk\, da\,dn$\,.
Moreover, if $U$ is a Lie group and $P$ is a closed subgroup of $U$, with left Haar measures 
$du$ and $dp$, respectively, then the $U$-invariant measure $d(uP)$ on the homogeneous space 
$U/P$ (when it exists) is normalized by 
\begin{equation}
\int_U f(u)\;du=\int_{U/P} \left( \int_P f(up)\; dp\right)\,d(uP)\,. 
\end{equation}
This condition normalizes the $G$-invariant measure $dx=d(gK)$ on $X=G/K$ and fixes the $K$-invariant 
measure $db=d(kM)$ on $B=K/M$ to have total mass 1.

Corresponding to the Iwasawa and the polar decompositions, we have the integral formulas 
\begin{align}
\int_X f(x)\; dx&=\int_G f(g\cdot o) \; dg =\int_K \int_A \int_N f(kan \cdot o)
e^{2\rho(\log a)} \; dk\, da\, dn  \label{eq:intformulaIwasawa}\\
                &=c \int_K \int_{A^+} f(ka \cdot o) \delta(a)\; dk\,da\, \label{eq:intformulapolar} 
\end{align}
where 
\begin{equation} \label{eq:delta}
\delta(\exp H)=\prod_{\a\in \Sigma^+} (\sinh \a(H))^{m_\a}
=e^{2\rho(H)} \prod_{\a\in \Sigma^+} \big(\frac{1-e^{-2\a(H)}}{2}\big)^{m_\a}
\,, \qquad H \in \frak a^+
\end{equation}
and $c$ is a suitable positive constant.

\subsection{Invariant differential operators on $X$ and the Laplace-Beltrami operator}
Let $\mathbb D(X)$ denote the commutative algebra of invariant differential operators on $X$ which are invariant under the action of $G$ on $X$ by left translations. 
As a Riemannian manifold, a symmetric space of the noncompact type $X=G/K$ is endowed with a Laplace-Beltrami operator $\Delta$ defined by $\Delta f=\diver \circ \grad f$ for $f \in C^\infty(X)$.
It turns out that $\Delta$ is a self-adjoint differential operator on $L^2(X)$ belonging to $\mathbb D(X)$. 

Let $\l \in \mathfrak a_\C^*$ and $b \in B$. Then the function $e_{\l,b}:X \to \C$ defined by
\begin{equation} 
e_{\l,b}(x)=e^{(i\l +\rho)(A(x,b))}\,, \qquad x \in X\,,
\end{equation}
is an eigenfuntion of $\Delta$. In fact, 
\begin{equation} \label{eq:Deltae}
\Delta e_{\l,b}=-(\inner{\l}{\l}+\inner{\rho}{\rho})e_{\l,b}\,.
\end{equation}  
We refer to \cite{He2}, Ch. II, 
for additional information. 

\subsection{$K$-invariant functions} \label{subsection:invariantfunctions}
Functions $f:X\to \C$ can be identified with right $K$-invariant functions on $G$, i.e. with 
functions $f:G\to \C$ so that $f(gk)=f(g)$ for all $g \in G$ and $k \in K$. Likewise, $K$-invariant functions on 
$X$ can be identified with $K$-bi-invariant functions on $G$, i.e. with 
functions $f:G\to \C$ so that $f(k_1 g k_2)=f(g)$ for all $g\in G$ and $k_1,k_2 \in K$. 
In turn, because of the polar decomposition $G=KAK=K\overline{A^+}K$, 
a $K$-bi-invariant function on $G$ can be identified with its $W$-invariant restriction to $A$
or with its restriction to $\overline{A^+}$. 
The Riemannian distance from the origin $o$, defined by $\sigma(x)=d(x,o)$, is an 
example of $K$-invariant map on $X$. 
According to the above identifications, we shall sometimes write $\sigma(g)$ instead of $\sigma(g\cdot o)$ for $g \in G$. 
If $x=g \cdot o$ and $g=k_1 \exp H
k_2$ with $g \in G$, $H \in \frak a$ and $k_1, k_2 \in K$, then 
\begin{equation}\label{eq:sigma}
\sigma(x)=\sigma(g)=\sigma(\exp H)=|H|. 
\end{equation}
The function $\sigma$ satisfies the following inequalities:
\begin{align}
\label{eq:trsigma}
\sigma(gh)&\leq \sigma(g)+\sigma(h)\,, \qquad g,h \in G\,,\\ 
\label{eq:ineqsigma}
\sigma(an)&\geq \sigma(a)\,, \qquad a \in A\,, n \in N\,.
\end{align}

Another $K$-invariant function on $X$ we shall employ is the spherical function $\Xi$ of spectral parameter $0$.
It is defined by 
\begin{equation}\label{eq:Xi}
 \Xi(x)=\int_B e_{0,b}(x)\, db =\int_K e^{-\rho(H(gk))}\; dk\,, \qquad x=g\cdot o \in X.
\end{equation}

It is a real analytic function and, writing $\Xi(g)$ instead of $\Xi(g\cdot o)$, it satisfies the following properties:
\begin{align}
&0 < \Xi(g)=\Xi(g^{-1}) \leq 1\,, \qquad g \in G\,, \label{eq:initialestXi}\\
&\int_K \Xi(g_1 k g_2)\; dk=\Xi(g_1)\Xi(g_2)\,, \qquad g_1\,,g_2 \in G\,, \label{eq:functeqXi}\\
&e^{-\rho(H)} \leq \Xi(a)\leq e^{-\rho(H)} (1+|H|)^d\,, \qquad a=\exp H \in \overline{A^+}\,, \label{eq:estimateXi}
\end{align}
where $d=|\Sigma^+_0|$ is the cardinality of the set of positive indivisible roots. 
The properties of the functions $\sigma$ and $\Xi$ can be found in \cite{GV}, \S 4.6 and \S 6.2.

\subsection{The Helgason-Fourier transform and the Radon transform}
\label{subsection:Helgason-Fourier-Radon}

A reference for this section is \cite{He3}, Ch. III.
The Helgason-Fourier transform of a sufficiently regular function $f:X\to \C$ 
is the function $\mathcal Ff$ defined by
\begin{equation} \label{eq:F-AN}
\mathcal F{f}(\lambda,b)=\int_X f(x) e_{-\l,b}(x) \; dx=\int_{AN} f(kan\cdot o) e^{(-i\l+\rho)(\log a)}\; da\,dn
\end{equation}
for all $\l\in \mathfrak a_\C^*$ and $b=kM\in B$ for which this integral exists.
The Plancherel theorem states that the Helgason-Fourier transform $\mathcal F$
extends to an isometry of $L^2(X)$ onto $L^2(\frak a^*_+ \times B, |c(\l)|^{-2} db\,d\l)$. 
The function $c(\l)$ occurring in the Plancherel density is Harish-Chandra's $c$-function.
It is the meromorphic function on $\frak a_\C^*$ given explicitly by the Gindikin-Karpelevich product formula:
\begin{equation}
\label{eq:c}
c(\l)=c_0 \prod_{\a\in \Sigma_0^+} c_\a(\l)
\end{equation}
where
\begin{equation}
\label{eq:ca}
c_\a(\l)= \frac{2^{-i\la} \; \Gamma(i\la)}
{\Gamma\Big(\frac{i\la}{2}+\frac{m_\a}{4}+\frac{1}{2}\Big)
\Gamma\Big(\frac{i\la}{2}+\frac{m_\a}{4}+\frac{m_{2\a}}{2}\Big)}
\end{equation}
and the constant $c_0$ is given by the condition $c(-i\rho)=1$.
In (\ref{eq:ca}) we have employed the notation
\begin{equation}\label{eq:la}
\l_\a:=\frac{\inner{\l}{\a}}{\inner{\a}{\a}}\,
\end{equation}
for $\a\in \Sigma$ and $\l\in\frak a_\C^*$.

Observe that if $f \in L^2(X)$ then, by the Plancherel theorem, for almost all $\l \in \mathfrak a^*$ we have 
$\mathcal Ff(\l,\cdot) \in L^2(B)$.

The Helgason-Fourier transform $\mathcal Ff$ of a function $f \in L^1(X)$ is almost everywhere defined. 
More precisely, there exists a subset $B' \subset B$ (depending on $f$) with $B \setminus B'$ of zero measure, such that $\mathcal Ff(\l,b)$ is defined for each $\l \in \frak a^*$. In fact, $\mathcal Ff(\l,b)$ is defined and holomorphic for $\l$ in a small tube domain around $\frak a^*$ in $\frak a_\C^*$.
It turns out that $\mathcal F$ is injective on $L^1(X)$. Furthermore, 
as in the Euclidean case, there is an inversion formula for functions $f \in L^1(X)$ with 
$\mathcal F f \in L^1(\frak a^* \times B, |c(\l)|^{-2} db\,d\l)$: for almost all $x \in X$,
we have 
\begin{equation} \label{eq:inversionF}
f(x)=\frac{1}{|W|} \int_{\frak a^* \times B} \mathcal Ff(\l,b) e_{\l,b}(x) \;
\frac{db\,d\l}{|c(\l)|^{2}}\,,
\end{equation}
where $|W|$ denotes the cardinality of the Weyl group $W$. 

By identifying the space $G/MN$ of horocycles on $X$ with $B \times A$, one can define the 
Radon transform of a 
sufficiently regular function $f:X\to \C$ as the function $Rf$ defined by
\begin{equation}
Rf(b,a)=e^{\rho(\log a)}\int_N f(kan \cdot o)\; dn
\end{equation}
for all $b=kM\in B=K/M$ and $a \in A$ for which this integral exists. 
See \cite{He3}, p. 220 or \cite{Eguchi-Lectures}. If $f \in L^1(X)$, 
then the integral defining $Rf(b,a)$ converges absolutely for almost all 
$(b,a) \in B \times A$ and $Rf \in L^1(B \times A)$. Indeed 
$\|Rf\|_{L^1(B \times A, db\,da)}\leq |W| \|f\|_1$; see \cite{Eguchi-Lectures}, Lemma 3.2.
By \cite{He3}, Ch. II, Theorem 3.2, the Radon transform is injective on $L^1(X)$.
Notice that if $Rf \in L^1(B \times A)$ then for almost all $a \in A$ we have $Rf(\cdot,a) \in L^1(B)$.

The Helgason-Fourier transform of a sufficiently regular function $f$ is the Euclidean Fourier transform of the Radon 
transform. For instance, if $f\in L^1(X)$, then  
\begin{equation}
\mathcal Ff(\l,b)=\mathcal F_A \big(Rf(b,\cdot)\big)(\l)=\int_A Rf(b,a) e^{-i\l(\log a)} \;da\,.
\end{equation}
for almost all $b\in B$ and all $\l \in \frak a^*$.
See e.g. the proof of formula (43), Ch. III, \S 1 in \cite{He3}.

\subsection{Analysis on $B=K/M$} 
\label{subsection:widehatK}

Let $\widehat{K}$ denote the set of (equivalence classes of) irreducible unitary representations of $K$. Fix $\delta \in \widehat{K}$ acting on the space $V_\delta$ of finite dimension $d(\delta)$. We consider $V_\delta$ endowed with an inner product $\inner{\cdot}{\cdot}$ making $\delta$ unitary. 
Let $M=Z_K(A)$ be the centralizer of $A$ in $K$. A vector $v \in V_\delta$ is said to be $M$-fixed if $\delta(m)v=v$ for all $m \in M$. Let $V_\delta^M$ denote the subspace of $M$-fixed vectors of $V_\delta$. We denote by 
$\widehat{K}_M$ the subset of $\widehat{K}$ consisting of (equivalence classes of) representations of $K$ for which $\dim V_\delta^M>0$. 

As before, let $B=K/M$. For $f:B \to \C$ sufficiently regular, we define $f^\delta: B \to \Hom(V_\delta,V_\delta)$
by 
\begin{equation}\label{eq:fdelta}
f^\delta(kM)=d(\delta) \int_K \delta(k_1^{-1}) f(k_1 kM)\; dk_1\,, \qquad k \in K\,.
\end{equation}
Then 
\begin{equation}\label{eq:fdeltabis}
f^\delta(kM)=\delta(k)f^\delta(eM)
\end{equation}
for all $k \in K$.
In particular, $f^\delta=0$ for $\delta \notin \widehat{K}_ M$.

The functions $f^\delta$ determine $f \in L^2(B)$  by the Peter-Weyl theorem for vector-valued functions, stating that 
\begin{equation} \label{eq:vectorPeterWeyl}
f=\sum_{\delta\in \widehat{K}_M} {\mathrm{Trace}}(f^\delta)\,.
\end{equation}
See \cite{He2}, Ch. V, Corollary 3.4. More generally, (\ref{eq:vectorPeterWeyl}) holds in the sense of distributions when $f \in L^1(B)$. See \cite{He2}, p. 508. Let  $v_1,\dots, v_{d(\delta)}$ be an orthonormal basis of $V_\delta$. For $i,j=1,\dots,d(\delta)$, define $f^\delta_{i,j}\in \C$ by
\begin{equation}\label{eq:deltafij}
f^\delta_{i,j}=d(\delta)^{-1} \inner{f^\delta(eM)v_i}{v_j}= \int_K \inner{\delta(k^{-1})v_i}{v_j} f(kM)\; dk\,.
\end{equation}
Observe that, by (\ref{eq:fdeltabis}) and (\ref{eq:deltafij}), we have $f^\delta=0$ if and only if $f^\delta_{i,j}=0$ for all 
$i,j=1,\dots,d(\delta)$. Moreover, by (\ref{eq:vectorPeterWeyl}), $f=0$ if and only if $f^\delta=0$ for all $\delta \in \widehat{K}_M$.

\section{The Euclidean case}
\label{section:Euclidean}

\subsection{The damped Schr\"odinger equation on $\R^n$}

Let $\mathcal F_{\R^n}$ denote the Fourier transform on $\R^n$. For a (sufficiently regular) 
function $f$ on $\R^n$, we also write $\widehat{f}=\mathcal F_{\R^n}f$.
We fix as a measure on $\R^n$ the Lebesgue measure divided by the factor $(2\pi)^{n/2}$. 

Consider the initial value problem for the time-dependent damped Schr\"odinger equation on $\R^n$:
\begin{equation} 
\label{eq:Schroedinger}
\tag{${\rm S}_c$}
\begin{split}
&i\partial_t u(t,x)+(\Delta -c)u(t,x)=0\\
&u(0,x)=f(x)
\end{split}
\end{equation}
where $\Delta$ denotes the Laplace operator on $\R^n$ and $c\in \R$ is the damping parameter.
Suppose $f \in L^2(\R^n)$. Then there is a unique $u \in C(\R:L^2(\R^n))$ satisfying (\ref{eq:Schroedinger}) in the sense of distributions and such that $u(0,\cdot)=f$. 
In fact, let $|\l|$ denote the euclidean norm of $\l \in \R^n$.
The Fourier transform $\mathcal F_{\R^n}$
gives a unitary equivalence of $\Delta$ with the multiplication operator $M$ on $L^2(\R^n)$
defined by $(Mf)(\l)=-|\l|^2 f(\l)$. It follows that the unique solution 
$u_t(x)=u(t,x)$ of (\ref{eq:Schroedinger}) is characterized by the equation

\begin{equation} \label{eq:functionaleqsolutionSR}
\widehat{u_t}(\l)= e^{-i(|\l|^2+c)t} \widehat{f}(\l)\,.
\end{equation}
Equivalently, $u_t=\mathcal F_{\R^n}^{-1}\big( e^{-i(|\l|^2+c)t} \mathcal F_{\R^n} f\big)$.
See for instance \cite{Rauch}, \S 3.5, Theorem 1 and Corollary 2.

We will need precise information on the solution of (\ref{eq:functionaleqsolutionSR}) in the space $\mathcal S'(\R^n)$ of tempered distributions on $\R^n$ for functions $f$ which are not necessarily in $L^2(\R^n)$. 

Let $O_M(\R^n)$ denote the space of functions $\varphi$ on $\R^n$ which are $C^\infty$ with slow growth (i.e. each derivative of $\varphi$ is bounded by a polynomial), and let 
 $O'_C(\R^n)$ be the space of rapidly decreasing distributions $T$ on $\R^n$ (i.e.  $T \in \mathcal S'(\R^n)$ and $T \ast \varphi\in 
\mathcal S(\R^n)$ for all $\varphi \in \mathcal S(\R^n)$).
See \cite{Schwartz}, 
p. 243--245. 
Then the Fourier transform is an isomorphism of $O_M(\R^n)$ onto $O'_C(\R^n)$.
For every fixed $a \in \R$, we have $e^{-ia|x|^2}\in O_M(\R^n) \cap O'_C(\R^n)$. 
The following properties hold. 

\begin{Lemma}
\label{lemma:convSchwartz}
If $S \in O_M(\R^n)$, $T \in O'_C(\R^n)$ and $F \in \mathcal S'(\R^n)$, then
\begin{enumerate}
\thmlist
\item $SF \in \mathcal S'(\R^n)$ and $T \ast F \in \mathcal S'(\R^n)$\,.
\item $\widehat{SF}=\widehat{S}\ast\widehat{F}$ and $\widehat{(T \ast F)}=\widehat{T}\widehat{F}$\,.
\end{enumerate}
\end{Lemma}
\begin{proof}
See \cite{Schwartz}, Ch. VII, \S 8, th\'eor\`eme XV.
\end{proof}

Let $c\in \R$ and $t\neq 0$ be fixed. 
In $\mathcal S'(\R^n)$, we have, with our normalisation of the Lebesgue measure:
$$\widehat{\gamma_{c,t}}(\l)=e^{-i(|\l|^2+c)t}$$
if
$$\gamma_{c,t}(x)=(2|t|)^{-n/2} e^{-ict}e^{-\pi i(\sign t)n/4} e^{i\frac{|x|^2}{4t}}\,,$$
where $\sign t:=t/|t|$.
See for instance \cite{HoermBook}, Theorem 7.6.1. Thus part (b) in Lemma \ref{lemma:convSchwartz}
gives for $F \in \mathcal S'(\R^n)$
$$e^{-i(|\l|^2+c)t} \widehat{F}=\widehat{\gamma_{c,t}}\widehat{F}=(\gamma_{c,t} \ast F)^\wedge\,.$$
Since the Fourier transform is injective on $\mathcal S'(\R^n)$, we obtain that the unique solution to the equation 
$$\widehat{u_t}=e^{-i(|\l|^2+c)t} \widehat{f}$$
in $\mathcal S'(\R^n)$ is 
\begin{equation}\label{eq:utgamma}
u_t=\gamma_{c,t} \ast f\,.
\end{equation}

Applying this property to the damped Schr\"odinger equation, we obtain in particular that, if $f \in L^1(\R^n)$, then for $t \neq 0$ the solution to (\ref{eq:Schroedinger}) can be written explicitly as 
\begin{equation}\label{eq:uh}
u(t,x)=(\gamma_{c,t} \ast f)(x)=(2\pi)^{n/2} 
(2|t|)^{-n/2} e^{-ict}e^{-\pi i(\sign t)n/4} e^{i\frac{|x|^2}{4t}} 
\; \widehat{h_t}\big(\frac{x}{2t}\big) 
\end{equation}
where 
\begin{equation}\label{eq:fh}
h_t(y)=e^{i\frac{|y|^2}{4t}}\; f(y)\,.
\end{equation} 
If $f \in L^2(\R^n)$, then (\ref{eq:uh}) and (\ref{eq:fh}) hold as equalities in $L^2(\R^n)$.

\subsection{Beurling-type uncertainty principles in $\R^n$}
\label{subsection:BeurlingEuclidean}
Beurling's theorem, proven by H\"ormander in \cite{Hoerm}, provides one of the strongest quantitative versions
of the uncertainty principles for the Fourier transform on $\R$.  
\begin{Thm}[Beurling's theorem] \label{thm:BeurlingR}
Let $f \in L^1(\R)$. If 
$$\int_\R \int_\R |f(x)||\widehat{f}(y)| e^{|xy|}\; dx\,dy <\infty$$
then $f=0$ almost everywhere.
\end{Thm}
A simplified proof of Theorem \ref{thm:BeurlingR} can be found in the appendix of \cite{BDJ}.
Some higher dimensional versions of this theorem have been proven recently by Bagchi and Ray. 
One of these versions is the following theorem. See \cite{BR}, Corollary 1. 
\begin{Thm}\label{thm:BR}
 Let $f \in L^1(\R^n)$. Suppose that 
$$\int_{\R^n} \int_{\R^n} |f(x)||\widehat{f}(y)| e^{|x||y|}\; dx\,dy <\infty\,.$$
Then $f=0$ almost everywhere.
\end{Thm}

\begin{Cor}
\label{cor:BeurlingSchrLuno}
Suppose $f, u_t$ satisfy (\ref{eq:functionaleqsolutionSR}). If $f \in L^1(\R^n)$ and 
there is $t_0> 0$ so that
\begin{equation}\label{eq:BeurlingSR}
\int_{\R^n}\int_{\R^n} |f(x)| |u(t_0,y)|  e^{\frac{|x||y|}{2t_0}} \; dx\,dy < \infty
\end{equation}
then $f=0$. Hence $u(t,\cdot)=0$ for all $t \in\R$. 
\end{Cor}
\begin{proof}
According to (\ref{eq:BeurlingSR}) and (\ref{eq:uh}), we have
\begin{align*}
+\infty > \int_{\R^n}\int_{\R^n} |f(x)| |u(t_0,y)| e^{\frac{|x||y|}{2t_0}} \; dx\,dy&= 
\frac{1}{(2t_0)^{n/2}} \, \int_{\R^n}\int_{\R^n} \big| h_{t_0}(x) \widehat{h_{t_0}}\big(\frac{y}{2t_0}\big)\big|  e^{\frac{|x||y|}{2t_0}} \; dx\,dy\\
&= (2t_0)^{n/2}\, \int_{\R^n}\int_{\R^n} |h_{t_0}(x)| |\widehat{h_{t_0}}(y)| e^{|x||y|} \; dx\,dy
\end{align*}
Theorem \ref{thm:BR} gives then $h_{t_0}=0$ and hence $f=0$. Thus $u(t,\cdot)=0$ for all $t\in \R$
by (\ref{eq:uh}).
\end{proof}

\begin{Cor}\label{cor:BeurlingSchrR}
Let $u(t,x) \in C(\R:L^2(\R^n))$ be the solution to the initial value problem 
(\ref{eq:Schroedinger}) with initial condition
$f \in L^2(\R^n)$. If there is $t_0> 0$ so that (\ref{eq:BeurlingSR}) holds.
Then $f=0$ and hence $u(t,\cdot)=0$ for all $t \in\R$. 
\end{Cor}
\begin{proof}
By Corollary \ref{cor:BeurlingSchrLuno}, it suffices to prove that $f \in L^1(\R^n)$.
If $u(t_0,\cdot)=0$ a.e., then $\widehat{f}=e^{i(|\l|^2+c)t_0} \widehat{u_{t_0}}=0$. Hence $f=0$. We can therefore suppose that $u(t_0,\cdot)$ is not a.e. zero.
By (\ref{eq:BeurlingSR}), we have
 \begin{equation} \label{eq:finiteBR}
|u(t_0,y)| \int_{\R^n} |f(x)| e^{\frac{|x||y|}{2t_0}} \; dx < \infty
\end{equation}
for almost all $y\in \R^n$. Since $u_{t_0}$ does not vanish almost everywhere, (\ref{eq:finiteBR}) gives for some $y_0\in \R^n$  
$$\int_{\R^n} |f(x)| \;dx\leq \int_{\R^n} |f(x)| e^{\frac{|x||y_0|}{2t_0}} \; dx < \infty\,.$$
Thus $f \in L^1(\R^n)$.
\end{proof}

\section{The case of a Riemannian symmetric space of the noncompact type}
\label{section:SchroedingerX}

\subsection{The Schr\"odinger equation on $X$}
\label{subsection:SchroedingerX}

Consider the initial value problem for the time-dependent Schr\"odinger equation on the Riemannian symmetric space $X=G/K$:
\begin{equation} 
\tag{${\rm S}$}
\begin{split}
&i\partial_t u(t,x)+\Delta u(t,x)=0\\
&u(0,x)=f(x)
\end{split}
\end{equation}
where $\Delta$ is the Laplace-Beltrami operator on $X$.
Suppose $f \in L^2(X)$. 
Then the analysis of solutions of the Schr\"odinger equation on $\R^n$ carries out to $X$ when the 
Fourier transform $\mathcal F_{\R^n}$ is replaced by the Helgason-Fourier transform $\mathcal F$.
In fact, by (\ref{eq:Deltae}), the function $e_{\l,b}$ appearing in the definition of the Helgason-Fourier transform are eigenfunctions of the Laplace-Beltrami operator $\Delta$. Hence, by the 
Plancherel theorem, $\mathcal F$ is a unitary equivalence of $\Delta$ with the multiplication operator
$M$ on $L^2(\mathfrak a^* \times B, \frac{d\l\, db}{|c(\l)|^2})$ defined by 
$(Mf)(\l,b)=-(|\l|^2+|\rho|^2) f(\l,b)$.
It follows that there is a unique $u_t(x)=u(t,x) \in C(\R:L^2(X))$ satisfying (\ref{eq:SchroedingerX}) in the sense of 
distributions on $X$ and such that $u(0,\cdot)=f$.
It is characterized by the equation
\begin{equation} \label{eq:FourierHelgasonrelation}
(\mathcal F u_t)(\l,b)= e^{-i(|\l|^2+|\rho|^2)t} \mathcal Ff(\l,b)\,.
\end{equation}
Equivalently, $u_t=\mathcal F^{-1}\big( e^{-i(|\l|^2+|\rho|^2)t} \mathcal F f\big)$.
We observe, in particular, that $u \in C(\R:\mathcal S(X))$ if $f \in \mathcal S(X)$. Here $\mathcal S(X)$ is the 
Schwartz space of smooth rapidly decreasing functions on $X$; see e.g. \cite{He3}, 
pp. 214--220. 

Notice that, as $\mathcal F$ is a bijection on $L^2(X)$, the condition $f=0$ implies $u_t=0$ for all $t\in \R$. Conversely, suppose there is $t_0\in \R$ so that $u_{t_0}=0$. Then (\ref{eq:FourierHelgasonrelation}) gives $e^{-i(|\l|^2+|\rho|^2)t_0} \mathcal Ff(\l,b)=0$ for all  $(\l,b)$. Hence $\mathcal Ff=0$. Thus $f=0$ and $u_t=0$ for all $t\in \R$. 
This remark applies more generally to solutions of (\ref{eq:FourierHelgasonrelation}) in $\mathcal S'(X)$, the space of tempered distributions on $X$.

\subsection{The class of functions $L^1(X)_C$}
\label{subsection:LunoC}

Let $\Xi$ and $\sigma$ be the functions defined in (\ref{eq:Xi}) and (\ref{eq:sigma}), respectively. Let $C\geq 0$. 
For a measurable function $h:X \to \C$ we set 
$$\norm{h}_{1,C}:=\int_X |h(x)|\Xi(x) e^{C\sigma(x)} \; dx $$
We denote by $L^1(X)_C$ the $\C$-vector space of (equivalence classes of a.e. equal) functions on $X$ for which $\norm{h}_{1,C}<\infty$.

The motivation for introducing this class of functions is the following. 
Suppose $f, u_{t_0}$ satisfy Beurling's condition (\ref{eq:BeurlingSchr}). Then $f \in L^1(X)_C$ for some $C \geq 0$. Indeed if $u(t_0,\cdot)$ vanishes almost everywhere, then $f=0$ by 
(\ref{eq:FourierHelgasonrelation}); if $u(t_0,\cdot)$ does not vanish almost everywhere, (\ref{eq:BeurlingSchr}) implies 
$$|u(t_0,y)|  \Xi(y) \int_X |f(x)| \Xi(x) e^{\frac{\sigma(x)\sigma(y)}{2t_0}} \; dx \,dy < \infty$$
for almost all $y \in X$, whence the result with $C=\sigma(y)/2t_0$ for any $y\neq o$ such that 
$u(t_0,y)\neq 0$. Likewise (\ref{eq:BeurlingSchr}) implies $u(t_0,\cdot) \in L^1(X)_{C'}$ for some $C'>0$, and we may assume $C=C'$.

\begin{Lemma} \label{lemma:basicLunoC}
Let $h$ be a measurable function on $X$ and let $C\geq 0$ be a constant. Then
\begin{equation}
\label{eq:LunoCG}
\int_X  |h(x)|\Xi(x) e^{C\sigma(x)} \; dx=\int_G |h(g\cdot o)|  e^{-\rho(H(g))} e^{C\sigma(g)} \; dg\,.
\end{equation}
Moreover, 
\begin{equation} \label{eq:RhunoC}
\int_B \int_A (R|h|)(b,a) e^{C\sigma(a)} \; da\, db \leq \norm{h}_{1,C}\,.
\end{equation}
Consequently, the Radon transform $Rh$ of a function $h \in L^1(X)_C$ is almost everywhere defined and $Rh \in L^1(B \times A, e^{C\sigma(a)}da\,db)$.
\end{Lemma}
\begin{proof}
Because of the definition of $\Xi$ in (\ref{eq:Xi}), we have
$$
\int_X  |h(x)|\Xi(x) e^{C\sigma(x)} \; dx=\int_G\int_K |h(g\cdot o)|  e^{-\rho(H(gk))} e^{C\sigma(g)} \; dg\,dk\,.$$
Then (\ref{eq:LunoCG}) follows by the right-$K$-invariance of $\sigma$ and the map 
$g \mapsto h(g \cdot o)$.

Suppose now that $\int_X  |h(x)|\Xi(x) e^{C\sigma(x)} \; dx<\infty$.
By (\ref{eq:LunoCG}), the integral formula (\ref{eq:intformulaIwasawa}) for the Iwasawa decomposition and the inequality (\ref{eq:ineqsigma}), we can write this integral as
\begin{align*}
\iiint_{K \times A \times N} &
|h(k a n\cdot o)| e^{-\rho(\log a)} e^{C\sigma(a n)} e^{2\rho(\log a)}  \; 
dk \, da\, dn\\
&\geq \int_K  \int_A \left(e^{\rho(\log a)} \int_N |h(k a n\cdot o)| \; dn\right)
e^{C\sigma(a)} \; 
da\, dk \\
&=\int_B \int_A (R|h|)(b,a) e^{C\sigma(a)} \; 
da\, db \,.
\end{align*}
This proves (\ref{eq:RhunoC}). 
The final property then follows as  
\begin{equation*}
|Rh(a,b)|\leq e^{\rho(\log a)} \int_N |h(kan\cdot o)| \; dn=R|h|(a,b)\,.
\end{equation*}
\end{proof}
The Radon transform $Rh$ of a function $h \in L^1(X)_C$ is well defined and $Rh \in L^1(B \times A)$. In fact, we also have $h \in L^1(X)$ when $C$ is sufficiently big. 

\begin{Lemma} \label{lemma:estimatesXisigma}
For all $x \in X$, we have 
\begin{equation}\label{eq:estimatesXisigma}
\Xi(x)e^{|\rho|\sigma(x)} \geq 1\,.
\end{equation}
Consequently, $L^1(X)_C\subset L^1(X)$ for $C \geq |\rho|$.
\end{Lemma}
\begin{proof}
By (\ref{eq:estimateXi}), for all $a=\exp H \in \overline{A^+}$ we have 
$\Xi(a\cdot o) e^{|\rho|\sigma(a\cdot o)} \geq \Xi(a\cdot o)e^{\rho(H)} \geq 1$.
The inequality (\ref{eq:estimatesXisigma}) then follows by $K$-invariance of $\Xi$ and $\sigma$ and by the decomposition $G=K\overline{A^+}K$. 
\end{proof}

Even if $h \in L^1(X)_C$ may not be in $L^1(X)$ unless $C$ is sufficiently big, we have 
$Rh(b,\cdot) \in L^1(A)$ for almost all $b \in B$ by Lemma \ref{lemma:basicLunoC}. So we can consider 
$\mathcal F_A \big( Rh(b,\cdot)\big)$. In fact, on $L^1(X)_C$ the equality $\mathcal F=\mathcal F_A \circ R$ holds. 
To prove this, we work on $K$-types. Let $\delta \in \widehat{K}_M$ and let $v_1,\dots, v_{d(\delta)}$ be a fixed orthonormal basis of the space $V_\delta$ of $\delta$. Recall from (\ref{eq:deltafij}) the notation $f_{i,j}^\delta$ for $f:B=K/M\to \C$. Recall also that
$f^\delta=0$ if and only if $f^\delta_{i,j}=0$ for all $i,j=1,\dots,d(\delta)$.

\begin{Lemma}\label{lemma:estimatesLunoC}
Let $C\geq 0$ be a constant and let $h \in L^1(X)_C$.
For $a \in A$ and $\l \in \mathfrak a^*$ set
\begin{align}
\label{eq:Rhij}
(Rh)_{i,j}^\delta(a)&:=((Rh)(\cdot,a))^\delta_{i,j}=
\int_K  \inner{\delta(k^{-1})v_i}{v_j} Rh(kM,a)\; dk\\
\label{eq:Fhij}
(\mathcal Fh)_{i,j}^\delta(\l)&:=(\mathcal Fh(\lambda,\cdot))_{i,j}^\delta=
\int_K  \inner{\delta(k^{-1})v_i}{v_j} \mathcal Fh(\l,kM)\; dk
\end{align}
For a function $g:\mathfrak a^* \to \C$, set $\norm{g}_\infty:=\sup_{\l \in \mathfrak a^*}
|g(\l)|$. Then 
\begin{align}
\label{eq:estFARhij}
\norm{\mathcal F_A \big((Rh)_{i,j}^\delta\big)}_\infty &\leq \norm{h}_{1,C}\\
\label{eq:estFhij}
\norm{(\mathcal Fh)_{i,j}^\delta}_\infty &\leq \norm{h}_{1,C}
\end{align}
Suppose moreover that $Rh \in L^1(B \times A)$ and set 
\begin{equation}
\label{eq:FRhij}
\big[ \mathcal F_A(Rh)\big]_{i,j}^\delta(\lambda):=
\int_K \inner{\delta(k^{-1})v_i}{v_j} \mathcal F_A\big( Rh(kM,\cdot)\big)(\l)\; dk\,.
\end{equation}
Then
\begin{equation}
\label{eq:FcircRhij}
\big[ \mathcal F_A(Rh)\big]_{i,j}^\delta=\mathcal F_A \big((Rh)_{i,j}^\delta\big)
\end{equation}
\end{Lemma}
\begin{proof}
By (\ref{eq:Rhij}) we have
\begin{equation} \label{eq:FARhijexpl}
\mathcal F_A \big((Rh)_{i,j}^\delta\big)(\l)=\int_A \int_K \inner{\delta(k^{-1})v_i}{v_j} Rh(kM,a) e^{i\l(\log a)} \; da\,dk\,.
\end{equation}
Since $\delta$ is unitary and $v_1,\dots,v_{d(\delta)}$ is an orthonormal basis, we obtain from (\ref{eq:RhunoC}):
$$|\mathcal F_A \big((Rh)_{i,j}^\delta\big)(\l)|\leq \int_A\int_B |Rh(b,a)|\;db\,da\leq 
\norm{h}_{1,C}\,.$$
This proves (\ref{eq:estFARhij}).
To prove (\ref{eq:estFhij}), we have by (\ref{eq:F-AN}):
\begin{align*}
|\mathcal Fh(\l,b)| &\leq \int_A \int_N |h(kan\cdot o)| e^{\rho(\log a)} \; da\, dn\\
&\leq \int_A \left( e^{\rho(\log a)} \int_N  |h(kan\cdot o)| \; dn\right) da\\
&=\int_A (R|h|)(b,a)\; da\,.
\end{align*}
Hence 
\begin{align*}
 |(\mathcal Fh)_{i,j}^\delta(\l)|&\leq \int_B  |\mathcal Fh(\l,b)|\; db\\
                                 &\leq \int_B \int_A (R|h|)(b,a)\; da\,db\,.
\end{align*}
So (\ref{eq:estFhij}) follows again from (\ref{eq:RhunoC}). Finally, (\ref{eq:FcircRhij}) is a consequence of (\ref{eq:FARhijexpl})
and Fubini's theorem, which applies as $Rh \in L^1(B \times A)$.
\end{proof}

\begin{Cor}
 \label{cor:FAR=Fijcompsupp}
Suppose $h \in C^\infty_c(X)$. Then for all $i,j=1,\dots,d(\delta)$ and $\l\in \mathfrak a^*$
\begin{equation}
 \big[\mathcal F_A(Rh)\big]_{i,j}^\delta(\l)=(\mathcal F h)_{i,j}^\delta(\l)\,.
\end{equation}
\end{Cor}
\begin{proof}
For $h \in C^\infty_c(X)$ we have $\mathcal F_A \circ R=\mathcal F$. 
\end{proof}

\begin{Prop}
\label{prop:FAR=F}
Let $C \geq 0$. Let $\delta\in \widehat{K}_M$ and 
$v_1,\dots,v_{d(\delta)}$ be as above. 
Then for all $h \in L^1(X)_C$
\begin{equation} \label{eq:FARij=Fij}
\mathcal F_A \big((Rh)_{i,j}^\delta\big)=\big[\mathcal F_A(Rh)\big]_{i,j}^\delta=(\mathcal Fh)_{i,j}^\delta
\end{equation}
as functions on $\mathfrak a^*$.
Consequently,
$\mathcal F=\mathcal F_A \circ R$ on $L^1(X)_C$. 
\end{Prop}
\begin{proof}
The first equality is a consequence of (\ref{eq:FcircRhij}). For the second,
let $h \in L^1(X)_C$, and let $h_n \in C_c^\infty(X)$ be a sequence converging to $h$ in 
$L^1(X)_C$. By (\ref{eq:estFARhij}), (\ref{eq:estFhij}) and (\ref{eq:FcircRhij}), for all $i,j=1,\dots, d(\delta)$, the sequences 
$[\mathcal F_A(Rh_n)]_{i,j}^\delta$ and $(\mathcal Fh_n)_{i,j}^\delta$ converge in $L^\infty(\mathfrak a^*)$ to $[\mathcal F_A(Rh)]_{i,j}^\delta$ and $(\mathcal Fh)_{i,j}^\delta$, respectively. But 
$[\mathcal F_A (Rh_n)]_{i,j}^\delta=(\mathcal Fh_n)_{i,j}^\delta$ by Corollary \ref{cor:FAR=Fijcompsupp}. 
Hence $[\mathcal F_A(Rh)]_{i,j}^\delta=(\mathcal Fh)_{i,j}^\delta$ by uniqueness of the limit.
Since this is true for all $\delta\in \widehat{K}_M$ and every orthonormal 
basis $v_1,\dots,v_{d(\delta)}$ of the space of $\delta$, we conclude that $\mathcal F=\mathcal F_A \circ R$ on $L^1(X)_C$. 
\end{proof}

We conclude this section by a remark on the convolution $h \times \psi$ of two elements in $L^1(X)_C$ under the additional assumption that $\psi$ is $K$-invariant. 
Recall that the convolution of two sufficiently regular functions $f_1$ and $f_2$ on $X$ is the function
$f_1 \times f_2$ defined on $X$ by $(f_1 \times  f_2) \circ \pi = (f_1 \circ \pi) \ast (f_2\circ \pi)$. Here $\pi : G \to X = G/K$ is the natural projection and $\ast$ denotes the convolution product of functions on $G$. This convolution is not commutative. 

\begin{Prop} \label{prop:convolutionEIC}
Let $C \geq 0$ and let $h, \psi \in L^1(X)_C$. Suppose that $\psi$ is $K$-invariant. Then $h \times \psi \in L^1(X)_C$. More precisely, we have
$$\int_X |(h\times \psi)(x)| \Xi(x) e^{C\sigma(x)} \; dx \leq \|h\|_{1,C} \|\psi\|_{1,C}
\,.$$
\end{Prop}
\begin{proof}
In the following we shall employ the same symbol to denote a function on $X$ and the corresponding $K$-invariant function on $G$. By definition of convolution products on $G$,
\begin{align*}
\int_G (|h| \ast |\psi|)(g) \Xi(g) e^{C\sigma(g)} \;dg &=
\int_G\int_G |h(g)| |\psi(g')|\Xi(gg')e^{C\sigma(gg')} \; dg\, dg'\\
&\leq \int_G \int_G |h(g)| |\psi(g')|\Xi(gg')e^{C\sigma(g)} e^{C\sigma(g')}\; dg\, dg'
\end{align*}
since $\sigma(gg')\leq \sigma(g)+\sigma(g')$ and $C\geq 0$. Replacing $g'$ by $kg'$, $k \in K$, and using the $K$-invariance of $\psi$ and $\sigma$, the latter integral becomes
$$\int_G \int_G |h(g)||\psi(g')| \Xi(gkg') e^{C\sigma(g)} e^{C\sigma(g')}\; dg\, dg'\,,
$$
which does not depend on $k$. Integration over $K$ leads to 
$$\int_G \int_G |h(g)||\psi(g')| \Xi(g)\Xi(g') e^{C\sigma(g)} e^{C\sigma(g')}\; dg\, dg'=
\|h\|_{1,C} \|\psi\|_{1,C}\,
$$
in view of the classical functional equation of spherical functions.
\end{proof}

\subsection{The Beurling-type condition for the Schr\"odinger equation on $X$}
\label{subsection:BeurlingX}

In this section, we shall prove Theorem \ref{thm:BeurlingSchr}.
Recall that if $u_t=u(t,\cdot)$ is a solution of (\ref{eq:SchroedingerX}) with initial condition $f \in L^2(X)$ satisfying
\begin{equation} \label{eq:B}
\int_X \int_X  |f(x)||u(t_0,y)| \Xi(x) \Xi(y) e^{\frac{\sigma(x)\sigma(y)}{2t_0}} \; dx \,dy <+\infty
\end{equation}
for some $t_0>0$, then $f, u_{t_0} \in L^1(X)_C$ for some constant $C>0$. 

\bigskip

\noindent \textit{Proof of Theorem \ref{thm:BeurlingSchr}.}\;
Observe that, by applying an argument similar to that in the proof of Lemma \ref{lemma:basicLunoC}, we get 
\begin{multline} \label{eq:BeurlingIntG}
\int_X \int_X  |f(x)||u(t_0,y)| \Xi(x) \Xi(y) e^{\frac{\sigma(x)\sigma(y)}{2t_0}} \; dx \,dy =\\
=\int_G \int_G |f(g_1\cdot o)||u(t_0,g_2\cdot o)|  e^{-\rho(H(g_1)+H(g_2))} e^{\frac{\sigma(g_1)\sigma(g_2)}{2t_0}} \; 
dg_1 \,dg_2. 
\end{multline}
By (\ref{eq:FourierHelgasonrelation}), to prove that $u(t,\cdot)=0$ for all $t\in \R$, it is enough 
to prove that $f=0$. But by Proposition \ref{prop:FAR=F}, we have $\mathcal Ff=\mathcal F_A(Rf)$
as $f \in L^1(X)_C$ for some $C>0$. Since $\mathcal F$ is an isometry on $L^2(X)$, it therefore suffices to prove that $Rf=0$. For this, we shall prove that $(Rh)_{i,j}^\delta=0$ for all $\delta \in \widehat{K}_M$ and $i,j=1,\dots,d(\delta)$.
This would complete the proof by the comments at the end of section \ref{subsection:widehatK}.

By proceeding as in the proof of (\ref{eq:LunoCG}), we have
\begin{align*}
&\int_X \int_X  |f(x)||u(t_0,y)| \Xi(x) \Xi(y) e^{\frac{\sigma(x)\sigma(y)}{2t_0}} \; dx \,dy \geq \\ 
& \qquad \geq
\int_A \int_B \int_A \int_B (R|f|)(b_1,a_1) (R|u_{t_0}|)(b_2,a_2) e^{\frac{|\log a_1||\log a_2|}{2t_0}} \; 
da_1\, da_2\,db_1 \, db_2 \\
& \qquad \geq \int_A \int_A \left(\int_B |Rf(b_1,a_1)| db_1\right)\left(\int_B |Ru_{t_0}(b_2,a_2)| db_2\right) e^{\frac{|\log a_1||\log a_2|}{2t_0}} \; da_1 \, da_2 \,.
\end{align*}
By (\ref{eq:B}) and (\ref{eq:Rhij}), this implies that 
\begin{equation}
\int_A \int_A |(Rf)^\delta_{i,j}(a_1)||(Ru_{t_0})^\delta_{i,j}(a_2)| 
e^{\frac{|\log a_1||\log a_2|}{2t_0}} \; 
da_1 \, da_2 < \infty
\end{equation} 
for all $\delta \in \widehat{K}_M$ and $i,j=1,\dots, d(\delta)$.

On the other hand, (\ref{eq:FourierHelgasonrelation}) together with (\ref{eq:FARij=Fij}) gives us
\begin{equation} \label{eq:SchroedingerAbelmatrix}
\mathcal F_A\big( (Ru_t)^\delta_{i,j}\big)(\l)= e^{-i(|\l|^2+|\rho|^2)t} 
\mathcal F_A\big( (Rf)^\delta_{i,j}\big)(\l)\,.
\end{equation}
Now, by Corollary \ref{cor:BeurlingSchrLuno}, we obtain
$$(Rf)^\delta_{i,j}=0$$
for all $\delta \in \widehat{K}_M$ and $i,j=1,\dots,d(\delta)$, concluding the result. 
\hfill\qed

\medskip

\begin{Rem}
Note that (\ref{eq:BeurlingIntG}) give us the Beurling's condition in group terms.
\end{Rem}

As an immediate corollary of Theorem \ref{thm:BeurlingSchr}, we obtain the following result for compactly supported initial conditions.

\begin{Cor} \label{cor:nocompsupport}
Let $u(t,x) \in C(\R:L^2(X))$ denote the solution of (\ref{eq:SchroedingerX}) with initial condition $f \in L^2(X)$. 
Suppose that $f$ has compact support. 
If there is a time $t_0> 0$ so that $u(t_0,\cdot)$ has compact support. 
Then $f=0$ and hence $u(t,\cdot)= 0$ for all $t \in \R$.
\end{Cor}
\begin{proof}
It suffices to observe that (\ref{eq:BeurlingSchr}) is always satisfied if $f$ and $u(t_0,\cdot)$ are compactly supported.
\end{proof}

\begin{Rem}
For $f \in L^2(X)$ the Radon transform $Rf(b,\cdot)$ as well as the $K$-types
$(Rf)^\delta_{i,j}$ appearing in (\ref{eq:SchroedingerAbelmatrix}) need not be in $L^2(A)$. This can be easily seen for functions 
$f$ satisfying $\mathcal Ff=\mathcal F_A(Rf)$. Indeed the Euclidean Fourier transform is an isometric isomorphism of $L^2(A)$ onto
$L^2(\mathfrak a^*)$. According to the Plancherel theorem, the image of $L^2(X)$ under the Helgason-Fourier transform is 
$L^2(\mathfrak a^*_+\times B,|c(\l)|^{-2} \; d\l\,db)$. So $Rf(b,\cdot) \in L^2(A)$ for almost all $b \in B$ provided 
$L^2(\mathfrak a^*,|c(\l)|^{-2} \; d\l) \subset L^2(\mathfrak a^*,d\l)$. The latter condition depends on the Harish-Chandra's $c$-function appearing in the Plancherel measure. 
Using the properties of the gamma function, one can prove the asymptotic behaviour
\begin{equation} \label{eq:asymptPlancherel}
\frac{1}{|c(\l)|^{2}} \asymp \prod_{\a\in\Sigma_0^+} |\inner{\l}{\a}|^2 \prod_{\a\in \Sigma_0^+} (1+|\inner{\l}{\a}|)^{m_\a+m_{2\a}-2}\,.
\end{equation}
See e.g. \cite{Anker}, Lemma 1. 
Here $f \asymp g$ means that there exists positive constants $C_1$ and $C_2$ so that $C_1 g(\l) \leq f(\l) \leq C_2 g(\l)$ for all 
$\l$. \end{Rem}

In the rank-one case, we have for instance the following result.

\begin{Lemma}
 \label{lemma:differentsLdue}
Suppose $\dim \mathfrak a^*=1$.
Let $h \in L^2(\mathfrak a^*, |c(\l)|^{-2} \; d\l)$ be bounded on $\{\l \in \mathfrak a^*: |\l|\leq r\}$ for some $r>0$. Then 
$h \in L^2(\mathfrak a^*, d\l)$.
\end{Lemma}
\begin{proof}
The asymptotic formula (\ref{eq:asymptPlancherel}) gives in this case: $|c(\l)|^{-2} \geq C$ for $|\l|\geq r$. 
If $h \in L^2(\mathfrak a^*, |c(\l)|^{-2} \; d\l)$, then $h$ is square integrable with respect to the Lebesgue measure on 
$\{\l \in \mathfrak a^*: |\l|\geq r\}$. Hence $h \in L^2(\mathfrak a^*, d\l)$, as it is bounded on the compact set 
$\{\l \in \mathfrak a^*: |\l|\leq r\}$.
\end{proof}

The boundedness of $\mathcal Ff(b,\cdot)$ at $\l=0$ for almost all $b \in B$ is obtained under very weak assumptions on $f$.
Recall for instance that for almost all $b \in B$, the function $\mathcal Ff(b,\cdot)$ is even holomorphic in a tube around
$\mathfrak a^*$. See section \ref{subsection:Helgason-Fourier-Radon}. Moreover, in this case $\mathcal Ff(b,\cdot)$ vanishes at infinity by Fatou's lemma for $\mathcal F$. We can then prove that $Rf\in L^2(A)$ for $f \in L^1(X) \cap L^2(X)$ whenever the root system satisfies the following conditions:

(C) \qquad Either there is no $\alpha \in \Sigma^+_0$ with multiplicity $m_\a=1$, or $2\alpha \in \Sigma^+$.

When condition (C) is met, then we have 
\begin{equation}\label{eq:est-partial-Plancherel-condition-C}
\prod_{\a\in \Sigma_0^+} (1+|\inner{\l}{\a}|)^{m_\a+m_{2\a}-2} \geq C
\end{equation}
 for all $\l \in \mathfrak a^*$. The result $Rf \in L^2(A)$ for $f \in L^2(X) \cap L^1(X)$ is then a consequence of the
above discussion and the following lemma. 

\begin{Lemma}
Let $\Sigma$ be a root system satisfying condition (C). 
Suppose $h \in L^2(\mathfrak a^*, |c(\l)|^{-2} \; d\l)$ is continuous and vanishes at infinity.  
Then $h \in L^2(\mathfrak a^*, d\l)$.
\end{Lemma}
\begin{proof}
Set $\Pi(\l)=\prod_{\a\in \Sigma_0^+} \inner{\l}{\a}$. By Lemma 5 in \cite{Anker}, 
the set $\Omega_1=\{\l\in\mathfrak a^*:|\Pi(\l)|\leq 1\}$ has finite Lebesgue measure.
Let $B_1=\{\l \in \mathfrak a^*:\|\l\| \leq 1\}$ and $C_1=\mathfrak a^* \setminus B_1$.
Write 
$$\int_{\mathfrak a^*} |h(\l)|^2 \; d\l =\int_{B_1} |h(\l)|^2 \; d\l + \int_{\Omega_1\cap C_1}  |h(\l)|^2 \; d\l + 
\int_{(\mathfrak a^* \setminus \Omega_1)\cap C_1}  |h(\l)|^2 \; d\l\,.$$
The first integral is finite as $h$ is continuous, hence bounded on the compact $B_1$. 
The second integral is also finite as $h$ is continuous and vanishes at infinity, so it is bounded in $C_1$, and $\Omega_1\cap C_1$ is a subset of $\Omega_1$, hence of finite measure. For the convergence of the third integral, we use condition (C). In fact, condition (C) yields 
(\ref{eq:est-partial-Plancherel-condition-C}). Moreover, on $\mathfrak a^* \setminus \Omega_1$, we have $|\Pi(\l)|\geq 1$. 
Thus $|c(\l)|^{-2} \geq C$ for all $\l \in (\mathfrak a^* \setminus \Omega_1)\cap C_1$.
Consequently, 
$$\int_{(\mathfrak a^* \setminus \Omega_1)\cap C_1}  |h(\l)|^2 \; d\l \leq 
C^{-1} \int_{(\mathfrak a^* \setminus \Omega_1)\cap C_1}  |h(\l)|^2 \; \frac{d\l}{|c(\l)|^2}\,< +\infty.$$
\end{proof}

\section{Applications}
\label{section:applications}

Let $X$ be a Riemannian symmetric space of the noncompact type. 
In this section we collect some uniqueness conditions for the solution of the Schr\"odinger equation (\ref{eq:SchroedingerX}) on 
$X$ which can be deduced from Theorem \ref{thm:BeurlingSchr}. They correspond to uncertainty principle conditions of 
Gelfand-Shilov type, Cowling-Price type and Hardy type. These results are parallel to the classical results of 
uncertainty principles for the Fourier transform on $\R^n$. Recall from (\ref{eq:initialestXi}) that 
$0 < \Xi(x) \leq 1$ for all $x \in X$.

\begin{Thm}[Gelfand-Shilov type] 
 \label{thm:GelfandShilov}
Let $u(t,x) \in C(\R : L^2(X))$ be the solution of (\ref{eq:SchroedingerX}) with initial condition $f \in L^2(X)$.
Suppose there exists positive constants $\alpha$, $\beta$ and a time $t_0>0$ so that
\begin{equation}\label{eq:GSXi}
\int_X |f(x)| \Xi(x) e^{\frac{\alpha^p}{p} \sigma^p(x)} \,dx < \infty \qquad \text{and} \qquad 
\int_X |u(t_0,x)| \Xi(x) e^{\frac{\beta^q}{q} \sigma^q(x)} \,dx < \infty 
\end{equation}
where $1<p<\infty$ and $\dfrac{1}{p}+\dfrac{1}{q}=1$\,. 
If $2t_0 \alpha\beta \geq 1$, then $f=0$ and hence $u(t,\cdot)=0$ for all $t\in \R$.
\end{Thm}
\begin{proof}
We have $\frac{\sigma(x)\sigma(y)}{2t_0}\leq  \alpha\beta \sigma(x)\sigma(y)\leq \frac{\alpha^p}{p} \sigma^p(x)+\frac{\beta^q}{q} \sigma^q(y)$. The inequalities (\ref{eq:GSXi}) imply then Beurling's condition (\ref{eq:BeurlingSchr}).
\end{proof}

\begin{Thm}[Cowling-Price type] 
 \label{thm:CowlingPrice}
Let $u(t,x) \in C(\R : L^2(X))$ be the solution of (\ref{eq:SchroedingerX}) with initial condition $f \in L^2(X)$.
Suppose there exists positive constants $a$, $b$ and a time $t_0>0$ so that
\begin{equation} \label{eq:CP}
\int_X \left( |f(x)| e^{a \sigma^2(x)}\right)^p \,dx < \infty \qquad \text{and} \qquad 
\int_X \left(|u(t_0,x)| e^{b \sigma^2(x)}\right)^q \,dx < \infty 
\end{equation}
where $1\leq p, q \leq \infty$. 
If $16t_0^2 ab > 1$, then $f=0$ and hence $u(t,\cdot)=0$ for all $t\in \R$.
\end{Thm}
\begin{proof}
Choose $A, B$ so that $0<A<a$, $0<B<b$ and $16t_0^2 AB > 1$. Let $p'$ and $q'$ so that 
$\dfrac{1}{p}+\dfrac{1}{p'}=1$ and $\dfrac{1}{q}+\dfrac{1}{q'}=1$.
Set $e_\eta(x)=e^{\eta \sigma^2(x)}$. Observe that, by 
(\ref{eq:intformulapolar}) and 
(\ref{eq:delta}), the function  $e_\eta \in L^p(X)$ for all $p \in [1,+\infty]$ if $\eta <0$.
Indeed, by $K$-invariance of $\sigma$, 
\begin{equation*}
\int_X |e_\eta(x)|^p \, dx \leq c \int_{\mathfrak a^+} e^{p\eta |H|^2} \; \delta(H)\, dH
 \leq c \int_{\mathfrak a^+} e^{p\eta |H|^2+2\rho(H)} \; dH <\infty\,.
\end{equation*}
By (\ref{eq:initialestXi}), H\"older inequality and the assumption, 
\begin{alignat*}{2}
&\normuno{f \Xi e_A} &&\leq \normuno{f e_A}\leq \|f e_a\|_p \|e_{(A-a)}\|_{p'} < \infty\,,\\ 
&\normuno{u(t_0,\cdot)\Xi e_B}
                      &&\leq \normuno{u(t_0,\cdot) e_B}
                      \leq
                      \|u(t_0,\cdot) e_b\|_q \|e_{(B-b)}\|_{q'} < \infty\,.
\end{alignat*}
The stated result is then a consequence of Theorem \ref{thm:GelfandShilov} with $p=q=2$, $A=\alpha^2/2$ and $B=\beta^2/2$. 
\end{proof}

\begin{Rem}
\label{remark:uncert-factors}
Let $\psi$ be a function on $X$ satisfying $\psi(x) e^{\nu \sigma^2(x)} \in L^\infty(X)$ and $\psi(x)^{-1} e^{\nu \sigma^2(x)} \in L^\infty(X)$ for all $\nu <0$. 
Because of the strict inequality $16t_0^2ab>1$ in Theorem \ref{thm:CowlingPrice}, we can replace the
functions $e^{\eta\sigma^2(x)}$, with $\eta \in \{a,b\}$, measuring the growth of $f$ and $u_{t_0}$, respectively, with $\psi(x) e^{\eta\sigma^2(x)}$.
For instance, we can choose $\psi(x)=\Xi(x)^M (1+\sigma(x))^N$ for some fixed integers 
$M$ and $N$. Similar remarks apply to Corollary \ref{cor:GelfandShilov} and Theorem \ref{thm:Hardy} below. 
\end{Rem}

\begin{Cor}
 \label{cor:GelfandShilov}
Let $u(t,x) \in C(\R : L^2(X))$ be the solution of (\ref{eq:SchroedingerX}) with initial condition $f \in L^2(X)$.
Suppose there exists positive constants $a$, $A$, $b$, $B$ and a time $t_0>0$ so that for all $x \in X$ 
\begin{equation}
|f(x)| \leq A\,  e^{-\frac{a^p}{p} \sigma^p(x)} \qquad \text{and} \qquad 
|u(t_0,x)| \leq B\, e^{-\frac{b^q}{q} \sigma^q(x)}  
\end{equation}
where $1<p<\infty$ and $\dfrac{1}{p}+\dfrac{1}{q}=1$\,. 
If $2t_0 ab > 1$, then $f=0$ and hence $u(t,\cdot)=0$ for all $t\in \R$.
\end{Cor}
\begin{proof}
Choose $\alpha,\beta$ so that $0<\alpha<a$, $0<\beta<b$ and $2t_0 \alpha\beta > 1$. Then $f$ and $u(t_0,\cdot)$ satisfy 
(\ref{eq:GSXi}).
\end{proof}

\begin{Thm}[Hardy type]
 \label{thm:Hardy}
Let $f$ be a measurable function on $X$ so that there exists positive constants $A$ and $\alpha$ so that 
\begin{equation} \label{eq:estfHardy}
|f(x)|\leq A \, e^{-\alpha\sigma^2(x)}
\end{equation} 
for all $x \in X$.
Then $f \in L^2(X)$. 
Let $u(t,x) \in C(\R : L^2(X))$ be the solution of (\ref{eq:SchroedingerX}) with initial condition $f$.
Suppose, moreover, that there is a time $t_0>0$ and positive constants $B$ and $\beta$ so that 
\begin{equation} \label{eq:estuHardy}
|u(t_0,x)|\leq B\, e^{-\beta \sigma^2(x)} 
\end{equation}
If $16 \alpha\beta  t_0^2 >1$, then $u(t,\cdot)= 0$ for all $t\in \R$.
\end{Thm}
\begin{proof}
By the $K$-binvariance of $\sigma$ and the integral formula (\ref{eq:intformulapolar}) and 
(\ref{eq:delta}), the growth condition on $f$ implies 
\begin{align*}
\int_X |f(x)|^2 \, dx &\leq A^2 \int_X  e^{-2\alpha\sigma^2(x)} \; dx 
\leq  c A^2 \int_{\mathfrak a^+} e^{-2\alpha|H|^2} \delta(H) \; dH \\
&\leq c A^2 \int_{\mathfrak a^+} e^{-2\alpha|H|^2+2\rho(H)} \; dH <\infty\,.
\end{align*}
Thus $f \in L^2(X)$. The uniqueness property is a consequence of Corollary \ref{cor:GelfandShilov} 
with $p=q=2$ and 
$\alpha=a^2/2$, $\beta=b^2/2$.
\end{proof}

\begin{Rem}
Under the additional assumptions that $f$ is $K$-invariant and $G$ is endowed with a complex structure, Theorem \ref{thm:Hardy} was proven in \cite{Chanillo} using the explicit expression of the elementary spherical functions. In fact, the Euclidean reduction via Radon transform provides an elementary proof of Theorem \ref{thm:Hardy} for arbitrary $X$ and $f$. 
Observe first that if $h$ is a measurable function on $X$ and there is $C >0$ for which 
$|h(x)|\leq A \, e^{-C\sigma^2(x)}$ for all $x \in X$, then $h \in L^1(X) \cap L^2(X)$.
Suppose now that $f$ and $u_{t_0}$ satisfy (\ref{eq:estfHardy}) and (\ref{eq:estuHardy}), respectively.
Choose $a,b$ so that $0< a < \alpha$,  $0< b < \beta$ and $16 ab t_0^2>1$.
By Proposition 1 in \cite{Sengupta2002}, there are constants $A', B'>0$ so that for all $(H,x)\in \mathfrak a \times B$ we have 
$|Rf(x,\exp H)|\leq A' \, e^{-a |H|^2}$ and $|Ru_{t_0}(x,\exp H)|\leq B' \, e^{-b |H|^2}$. 
By (\ref{eq:FourierHelgasonrelation}) and the fact that $\mathcal F=\mathcal F_A \circ R$ on $L^1(X)$, 
for all $x\in B$, we have that  
$Ru_t(x,\cdot)$ is the solution of the damped Schr\"odinger equation on $A \equiv \R^n$ with initial condition $Rf(x,\cdot)\in L^2(\R^n)$ and damping parameter $|\rho|^2$. The Hardy type uniqueness theorem for the Schr\"odinger equation on $\R^n$ (Theorem 2 in \cite{Chanillo}) yields then $Rf(x,\cdot)=0$ for all $x \in B$. Since $R$ is injective on $L^1(X)$, we conclude that $f=0$ 
and hence $u(t,\cdot)=0$ for all $t\in \R$.
\end{Rem}

In \cite{Chanillo}, the value $t_0$ given in Theorem \ref{thm:Hardy} by the inequality $16\alpha\beta t_0^2 >1$  was proven to be optimal. 
Indeed, for the symmetric space $X=SL(2,\C)/SU(2)$, Chanillo gave the following example. 
Let $f$ be the function on $X$ which agrees on the maximally flat geodesic submanifold $A\equiv \R$
with the function  $e^{-x^2-i x^2/4}$. Then $f$ satisfies (\ref{eq:estfHardy}) with $\alpha=1$. The solution $u(t,\cdot)$ of the Schr\"odinger equation (\ref{eq:SchroedingerX}) with initial condition $f$ 
is not identically zero even if it satisfies (\ref{eq:estuHardy}) for $\beta=1/16$ at $t_0=(16\alpha\beta)^{-1/2}=1$. Thus the uniqueness property fails in this case for some 
$\alpha$, $\beta$ and $t_0$ with $16\alpha\beta t_0^2=1$.

In the rest of this section we provide additional information on the optimality of $t_0$.
We consider some Hardy type uniqueness conditions which are slightly more restrictive than those stated in Theorem \ref{thm:Hardy}. So they still imply that the solutions of (\ref{eq:SchroedingerX}) are identically zero provided the condition $16\alpha\beta t_0^2 >1$
holds. We then characterize the non-unique solutions in the case $t_0=(16\alpha\beta)^{-1/2}$ under
the additional assumption that $G$ is endowed with a complex structure and the initial condition $f$ is $K$-invariant. 

Observe first that we can modify the right-hand side of the estimates in Theorem \ref{thm:Hardy} by a suitable positive bounded factor $\psi$ without loosing the uniqueness property when $16 \alpha\beta t_0^2 >1$. We choose here $\psi$ to be the function used by Harish-Chandra to control the decay of the elements in the $K$-biinvariant $L^2$-Schwartz space on $G$. See for instance \cite{GV}, p. 256. Considered as a $K$-biinvariant function on $G$, the function $\psi$ is uniquely defined on $G$ by the condition 
\begin{equation}\label{eq:D}
\psi(\exp H)=\prod_{\gamma\in \Sigma^+} \left(\frac{\gamma(H)}{\sinh \gamma(H)}\right)^{m_\gamma/2}\,, 
\qquad H \in \mathfrak a.
\end{equation}
Then $\psi$ is a positive $K$-invariant function on $X$ which is bounded. More precisely, one can prove that 
for suitable positive constants $c_1, c_2$ and nonnegative integers $d_1, d_2$ one has
\begin{equation*}
c_1 \Xi(x) (1+\sigma(x))^{-d_1} \leq \psi(x) \leq c_2 \Xi(x) (1+\sigma(x))^{d_2}\,
\end{equation*}
for all $x \in X$. The following result is then a consequence of Theorem \ref{thm:Hardy}.

\begin{Cor}\label{cor:sharpHardy}
Let $f$ be a measurable function on $X$ and assume there exist positive constants $A$ and $\alpha$ so that 
\begin{equation} \label{eq:sharpestfHardy}
 |f(x)|\leq A \, \psi(x) e^{-\alpha\sigma^2(x)}
\end{equation} 
for all $x \in X$.
Then $f \in L^2(X)$. 
Let $u(t,x) \in C(\R : L^2(X))$ be the solution of (\ref{eq:SchroedingerX}) with initial condition $f$.
Suppose, moreover, that there is a time $t_0>0$ and positive constants $B$ and $\beta$ so that 
\begin{equation} \label{eq:sharpestuHardy}
|u(t_0,x)|\leq B\, \psi(x) e^{-\beta \sigma^2(x)}\,. 
\end{equation}
If $16 \alpha\beta  t_0^2 >1$, then $u(t,\cdot)\equiv 0$ for all $t\in \R$.
\end{Cor}

Suppose now that $G$ has a complex structure and $f$ is a $K$-invariant function on $X$.
A $K$-invariant function on $X$ can be identified with a $W$-invariant function on 
$A=\exp \mathfrak a \equiv \mathfrak a$. Here, as before, $W$ denotes the Weyl group. 
In this identification, the space of $K$-invariant functions in $L^2(X)$ corresponds to the space of $W$-invariants in $L^2(\mathfrak a,\eta^2(H) \,dH)$, 
where 
\begin{equation}\label{eq:eta}
\eta(H)=\prod_{\gamma \in \Sigma^+}\sinh \gamma(H)\,,\qquad H \in \mathfrak a\,.
\end{equation}
Moreover, the radial component on $A^+\equiv \mathfrak a^+$ of the Laplace-Beltrami operator 
$\Delta$ on $X$ is 
\begin{equation}\label{eq:radialDelta}
\frac{1}{\eta(H)} \, \left( \Delta_{\mathfrak a} -|\rho|^2\right) \circ \eta(H)
\end{equation}
where $\Delta_{\mathfrak a}$ is the Laplace operator on $\mathfrak a$. If $n=\dim \mathfrak a$ and
$\{H_j\}_{j=1}^n$ is an orthonormal basis of $\mathfrak a$ with respect to the inner product 
induced by the Killing form, then 
$\Delta_{\mathfrak a}= \sum_{j=1}^n \partial(H_j)^2$.
Because of (\ref{eq:radialDelta}), $u_t=u(t,\cdot)$ is the solution to (\ref{eq:SchroedingerX}) with $K$-invariant initial condition $f \in L^2(X)$ if and only if $\eta(H)u_t(H)$ is the solution 
of the damped Schr\"odinger equation (\ref{eq:Schroedinger}) on $\mathfrak a\equiv \R^n$ with damping parameter $c=|\rho|^2$ and $W$-skew-invariant initial condition $\eta(H)f(H)$.

When $G$ admits a complex structure, all root multiplicities $m_\gamma$ are equal to $2$.
Writing $\psi(H)$ instead of $\psi(\exp H)$, we therefore have
\begin{equation*}
\psi(H)=\frac{\pi(H)}{\eta(H)}\,, \qquad H \in \mathfrak a\,,
\end{equation*}  
where 
\begin{equation*}
\pi(H)=\prod_{\gamma \in \Sigma^+} \gamma(H)\,.
\end{equation*}

Let $\alpha,\beta$ be two positive constants and let  
$$f(x)=\psi(x) e^{-\alpha \sigma^2(x)} e^{-i \sqrt{\alpha\beta} \sigma^2(x)}\,, \qquad x \in X\,.$$
Then $f$ is a $K$-biinvariant function in $L^2(X)$. Considered as a function on $A\equiv \mathfrak a$, we have
$$f(H)= \frac{\pi(H)}{\eta(H)}\, e^{-\alpha|H|^2} e^{-i\sqrt{\alpha\beta}|H|^2}\,, \qquad H \in \mathfrak a\,.$$
Let $u(t,x)$ be the solution of (\ref{eq:SchroedingerX}) with initial condition $f$.
Then $\eta(H)u(t,H)$\,, $H \in \mathfrak a$, is the solution of (\ref{eq:Schroedinger}) with damping parameter $|\rho|^2$ and initial condition $\eta(H)f(H)$. Set $t_0=(16\alpha\beta)^{-1/2}$. By (\ref{eq:uh}) and (\ref{eq:fh}) we have for a constant $C_1$ depending on $t_0$
$$
\eta(H)u(t_0,H)=C_1 e^{i\frac{|H|^2}{4t_0}} \widehat{h_{t_0}}\big(\frac{H}{2t_0}\big)
$$
where
$$h_{t_0}(Y)= \pi(Y) e^{-\alpha|Y|^2}\,.$$
Since
$$\widehat{h_{t_0}}(H)=\pi(i \partial) \big(e^{-\alpha|Y|^2}\big)^\wedge=C_2 \pi(i \partial)
e^{-\frac{|H|^2}{4\alpha}}$$
and $16t_0^2\alpha=\beta^{-1}$, we conclude that
$$
\eta(H)u(t_0,H)=C_3 e^{i\frac{|H|^2}{4t_0}} \pi(i \partial)  e^{-\beta|H|^2}
$$
where $C_3$ is a constant depending on $t_0$.
Moreover,
$\pi(i \partial)= \prod_{\gamma \in \Sigma^+} i \partial(H_\gamma)$, where
$H_\gamma$ is the unique element in $\mathfrak a$ so that $\gamma(H)=\inner{H_\gamma}{H}$ for all 
$H \in \mathfrak a$. Since $\partial(H_\gamma) e^{-\beta|H|^2}=-2\beta \gamma(H) e^{-\beta|H|^2}$, we have $\pi(i \partial) e^{-\beta|H|^2}=P(H) e^{-\beta|H|^2}$ where $P(H)$ is a polynomial of degree 
$\deg P=\deg \pi$. Observe that $P(H)e^{-\beta|H|^2}$ is a constant multiple of $e^{-i\frac{|H|^2}{4t_0}} \eta(H)u(t_0,H)$. Since $\eta(H)$ is $W$-skew-invariant, so is $P(H)$. Hence $\pi(H)$ divides 
$P(H)$. Thus $P(H)=c \pi(H)$ for a constant $c$. 
Therefore
$$u(t_0,H)=c_{t_0} \frac{\pi(H)}{\eta(H)} e^{-\beta|H|^2} e^{i\frac{|H|^2}{4t_0}}$$
for all $H \in \mathfrak a$, i.e.
$$u(t_0,x)=c_{t_0} \psi(x) e^{-\beta\sigma^2(x)} e^{i\frac{\sigma^2(x)}{4t_0}}$$
for all $x \in X$. 
Since $f$ and $u_t$ respectively satisfy the estimates (\ref{eq:sharpestfHardy}) and (\ref{eq:sharpestuHardy}), the uniqueness property in Corollary \ref{cor:sharpHardy} fails 
when $16 t_0^2 \alpha\beta=1$. 

In fact, we shall prove in Theorem \ref{thm:complexnonunique} below that the function $f$ we just considered is (up to constant multiples) the only function which can occur in the case $16t_0^2 \alpha\beta=1$. We shall derive this property from the following equality case of Hardy's theorem on $\R^n$; see \cite{Thangavelu}, Theorem 1.4.4.

\begin{Lemma} \label{lemma:HardyequalR}
Suppose $h$ is a measurable function on $\R^n$ that satisfies the estimates 
$$|h(x)|\leq C(1+|x|^2)^m e^{-a|x|^2}\qquad\textrm{and}\qquad 
|\widehat{h}(\xi)|\leq C(1+|\xi|^2)^m e^{-b|\xi|^2}$$
for some positive constants $a$ and $b$.
When $ab=\frac{1}{4}$, then $h(x)=P(x) e^{-a|x|^2}$ where $P$ is a polynomial of degree 
$\leq 2m$.
\end{Lemma}

\begin{Thm}\label{thm:complexnonunique}
Let $X=G/K$ be a Riemannian symmetric space with $G$ complex. Suppose $f$ is a $K$-invariant 
function on $X$ satisfying (\ref{eq:sharpestfHardy}). Let $u(t,\cdot)$ be the solution of 
(\ref{eq:SchroedingerX}) with initial condition $f$. Assume that $u(t_0,x)$ satisfies (\ref{eq:sharpestuHardy}). If $16\alpha\beta t_0^2=1$, then there exists a constant $C$  
so that  
\begin{equation*}
f(x)=C \, \psi(x) e^{-\alpha \sigma^2(x)} e^{-i \sqrt{\alpha\beta} \sigma^2(x)}\,.
\end{equation*}
for all $x \in X$.
\end{Thm}
\begin{proof}
It remains to prove that if $f$ and $u_{t_0}$ satisfy (\ref{eq:sharpestfHardy}) and (\ref{eq:sharpestuHardy}), respectively, then $f(x)$ is a constant multiple of $\psi(x) e^{-\alpha \sigma^2(x)} e^{-i \sqrt{\alpha\beta} \sigma^2(x)}$. 

We know that $\eta(H)u_t(H)$ is the solution of the damped Schr\"odinger equation on $\mathfrak a$ with damping parameter $|\rho|^2$ and initial condition $\eta(H)f(H) \in L^2(\mathfrak a,dH)$. Hence, by (\ref{eq:uh}), 
\begin{equation*}
\eta(H)u(t_0,H)=(\gamma_{|\rho|^2,t_0} \ast \eta \cdot f)(H)=C_{t_0} e^{i\frac{|H|^2}{4t_0}} \widehat{h_{t_0}}\big(\frac{H}{2t_0}\big)
\end{equation*}
where 
\begin{equation*}
h_{t_0}(H)=e^{i\frac{|H|^2}{4t_0}} \eta(H)f(H)\,.
\end{equation*}
Hence
\begin{equation*}
|h_{t_0}(H)|=|\eta(H)f(H)|=|\pi(H)\psi(H)^{-1} f(H)| \leq A |\pi(H)| e^{-\alpha|H|^2}
\end{equation*}
and
\begin{equation*}
\big|\widehat{h_{t_0}}\big(\frac{H}{2t_0}\big)\big| = C_{t_0}^{-1} |\pi(H)\psi(H)^{-1} u(t_0,H)| 
\leq B C_{t_0}^{-1} |\pi(H)| e^{-\beta |H|^2}\,.
\end{equation*}
So 
\begin{equation*}
|\widehat{h_{t_0}}(\xi)| \leq C'_{t_0} |\pi(\xi)| e^{-4 t_0^2\beta |\xi|^2}\,.
\end{equation*}
Let $m$ be the smallest integer such that $2m \geq |\Sigma^+|$. Since 
$|\pi(H)|\leq C(1+|H|^2)^m$ and $\alpha(4 t_0^2 \beta)=1/4$, we obtain from 
Lemma \ref{lemma:HardyequalR} that $h_{t_0}(H)=P(H) e^{-\alpha|H|^2}$ where $P(H)$ is a polynomial of degree
$\leq 2m$. Therefore
$e^{i\frac{|H|^2}{4t_0}} \eta(H)f(H)=P(H) e^{-\alpha|H|^2}$.
This equality shows that $P(H)$ must be $W$-skew-invariant, so divisible by $\pi(H)$.
Hence $P(H)=q(H)\pi(H)$ for a $W$-invariant polynomial $q(H)$ so that 
$\deg q=\deg P -\deg \pi\leq 2m -|\Sigma^+|\leq 1$. This is only possible when $q(H)=C$ is a constant.
Thus 
$f(H)=C \frac{\pi(H)}{\eta(H)} \, e^{-i\frac{|H|^2}{4t_0}} e^{-\alpha|H|^2}$, which proves the claim.
\end{proof}

\end{document}